\documentclass[11pt]{amsart}
\usepackage{amsmath,amssymb,amsthm,mathrsfs}
\usepackage[all]{xy}
\usepackage{mathtools}

\makeatletter
\@namedef{subjclassname@2020}{\textup{2020} Mathematics Subject Classification}
\makeatother

\theoremstyle{definition}
\newtheorem{Def}{Definition}[section]
\newtheorem{Thm}[Def]{Theorem}
\newtheorem{Prop}[Def]{Proposition}

\newtheorem{Ex}[Def]{Example}

\newtheorem{Rem}[Def]{Remark}

\newcommand{\C}{\mathbb{C}}
\newcommand{\R}{\mathbb{R}}
\newcommand{\Q}{\mathbb{Q}}

\newcommand{\Z}{\mathbb{Z}}
\newcommand{\N}{\mathbb{N}}
\newcommand{\PP}{\mathbb{P}}
\newcommand{\DD}{\mathfrak{D}}

\newcommand{\Gr}{\mathrm{Gr}}
\newcommand{\OO}{\mathrm{O}}
\newcommand{\SO}{\mathrm{SO}}

\newcommand{\rank}{\mathrm{rank}}
\newcommand{\Diff}{\mathrm{Diff}}

\newcommand{\Kt}{\mathrm{K3}}
\newcommand{\gKt}{\mathrm{gK3}}
\newcommand{\re}{\mathrm{Re}}
\newcommand{\im}{\mathrm{Im}}

\begin{document}
\title{Mirror symmetry and rigid structures of generalized K3 surfaces}
\author{Atsushi Kanazawa}

\maketitle

\begin{abstract}
The present article is concerned with mirror symmetry for generalized K3 surfaces, with particular emphasis on complex and K\"ahler rigid structures. 
Inspired by the works of Dolgachev, Aspinwall--Morrison and Huybrechts, we introduce a formulation of mirror symmetry for generalized K3 surfaces by using Mukai lattice polarizations. 
This approach solves issues in the conventional formulations of mirror symmetry for K3 surfaces. 
In particular, we provide a solution to the problem of mirror symmetry for singular K3 surfaces. 
Along the way, we investigate complex and K\"ahler rigid structures of generalized K3 surfaces. 
\end{abstract}


\section{Introduction}
The present article is concerned with mirror symmetry for generalized K3 surfaces, with particular emphasis on complex and K\"ahler rigid structures. 
A few important aspects of mirror symmetry for K3 surfaces are unified in the framework of generalized K3 surfaces. 

Mirror symmetry for a K3 surface $S$ is a very subtle problem, as the complex and K\"ahler structures are somewhat mixed in $H^2(S,\C)$. 
In the foundational article \cite{Dol}, Dolgachev formulated mirror symmetry for lattice polarized K3 surfaces as a duality between the algebraic and transcendental cycles.   
Although his formulation works well in many cases, it cannot be definitive due to an assumption that does not hold  in general. 
A notable exception is a singular K3 surface, which has the maximal Picard number $20$.  
This kind of K3 surfaces, also known as a rigid K3 surface, admits no deformation of complex structure while retaining the maximal Picard number. 
A satisfactory formulation of mirror symmetry for K3 surfaces, particularly one that solves the puzzle of mirror symmetry for singular K3 surfaces, has been anticipated.

The celebrated work \cite{AM} of Aspinwall--Morrison is a key article explaining mirror symmetry for K3 surfaces from a physics perspective. 
They discussed SCFTs on a K3 surface $S$ and demonstrated that the moduli space  $\mathfrak{M}_{(2,2)}$ of $N=(2,2)$ SCFTs fibers over the moduli space $\mathfrak{M}_{(4,4)}$ of $N=(4,4)$ SCFTs. 
Mathematically such moduli spaces are related with space-like 4-spaces in $H^*(S,\R) \cong \R^{4,20}$ equipped with the Mukai pairing, and  
$\mathfrak{M}_{(2,2)}$ and $\mathfrak{M}_{(4,4)}$ are identified with certain Grassmannians $\Gr^{po}_{2,2}(H^*(M,\R))\cong  \OO(4,20)/(\SO(2)\times \SO(2)\times \OO(20))$ 
and $\Gr^{po}_{4}(H^*(M,\R)) \cong \OO(4,20)/(\SO(4)\times \OO(20))$ respectively. 
Then mirror symmetry is realized as an involution of these moduli spaces. 
On the other hand, the moduli space $\mathfrak{M}_{\mathrm{HK}}$ of $B$-field shifts of hyperK\"ahler metrics 
can be identified with an open dense subset of $\mathfrak{M}_{(4,4)}$ by the period map $\mathfrak{per}_{\mathrm{HK}}$. 
Inspired by this fact, Huybrechts showed in \cite{Huy} that the counterpart of $\mathfrak{M}_{(2,2)}$ is given by the moduli space $\mathfrak{M}_{\gKt}$ of generalized K3 surfaces, 
which are the K3 version of the generalized Calabi--Yau structures introduced by Hitchin \cite{Hit}.  
The moduli space $\mathfrak{M}_{\Kt} \times H^2(M,\R)$ of $B$-field shifts of K3 surfaces endowed with a Ricci-flat metric is naturally contained in $\mathfrak{M}_{\gKt}$ of real codimension $2$. 
To summarize, there is the following diagram. 
$$
\xymatrix{
\mathfrak{M}_{\Kt} \times H^2(M,\R)   \ar@{^{(}->}[r]^-{\iota} &  \mathfrak{M}_{\gKt} \ar[d] \ar[r]^-{\mathfrak{per}_{\gKt}} &  \Gr^{po}_{2,2}(H^*(M,\R))=\mathfrak{M}_{(2,2)} \ar[d]  \\
 &  \mathfrak{M}_{\mathrm{HK}} \ar[r]^-{\mathfrak{per}_{\mathrm{HK}}}  &   \Gr^{po}_{4}(H^*(M,\R)) =\mathfrak{M}_{(4,4)}
}
$$
It is therefore natural to expect that points in $\mathfrak{M}_\Kt$ might be mirror symmetric to points that are no longer in $\mathfrak{M}_\Kt$. 
We will demonstrate that this is precisely the source of the problems in the conventional formulation of mirror symmetry for K3 surfaces (cf. Remark \ref{AM SCFT}). 

There is a good chemistry between Hitchin's generalized Calabi--Yau structures and mirror symmetry as they both embrace the symplectic and complex on an equal footing. 
Inspired by works of Dolgachev, Aspinwall--Morrison and Huybrechts, we introduce a formulation of mirror symmetry for generalized K3 surfaces (Section \ref{MS for gK3}).  
The key features of our formulation are twofold, and both are guided by Hitchin's theory. 
The first is to extend the scope of lattice polarizations to the Mukai lattice. 
Mixture of degrees of cycles is indispensable for generalized K3 surfaces. 
The second is to treat the two structures in mirror symmetry on completely equal footing. 
The N\'eron--Severi lattice and transcendental lattice are defined in the same fashion, and moreover the lattice polarizations are imposed on both. 

The primary advantage of our formulation is that the elements in $H^0(M,\Z)$ and $H^4(M,\Z)$ are no longer treated as special. 
This allows us to eliminate the artificial assumption present in Dolgachev's formulation of mirror symmetry. 
Indeed, the assumption is a reflection of the existence of the hyperbolic lattice $U \cong H^0(M,\Z) \oplus H^4(M,\Z)$ under mirror symmetry.  
Our formulation naturally settles the two fundamental problems in the conventional formulation (Section \ref{MS problem}).  

Along the way, we investigate the complex and K\"ahler rigid structures of generalized K3 surfaces (Section \ref{rigid structures}).  
The notion of a complex rigid structure is enhanced to the level of generalized K3 surfaces. 
On the other hand, the notion of a K\"ahler rigid structure naturally appears as a counterpart of complex rigid structure from the viewpoint of mirror symmetry. 
We define this notion by using a generalized Calabi--Yau structure, which captures an interesting sublattice of $H^*(M,\Z)$ derived from the K\"ahler moduli space (Section \ref{Kahler rigidity computation}).

 \subsection*{Structure of article}
 Section \ref{gK3} provides a brief summary of generalized K3 structures based on Huybrechts' work \cite{Huy}. 
 Section \ref{rigid structures} investigates the complex and K\"ahler rigid structures on generalized K3 surfaces. 
 Section \ref{MS for gK3} focuses on mirror symmetry for generalized K3 surfaces in comparison with the classical formulation of that for K3 surfaces.  
The main problems are described in Section \ref{MS problem} and they are solved in Section \ref{MS_B} and Section \ref{MS for sing K3}.

\subsection*{Acknowledgment}
The author sincerely thanks Emma Brakkee, Yu-Wei Fan, Kenji Hashimoto and Shinobu Hosono for their valuable discussions. 
Special gratitude is also extended to the referees for their helpful comments and suggestions, particularly regarding the references to \cite{SS}.   
This work is supported by the JSPS Grant-in-Aid Wakate(B) 17K17817, Kiban(C) 22K03296 and Kiban(A) 23H00083.



\section{Generalized K3 surfaces} \label{gK3}

Hitchin's invention of generalized Calabi--Yau structures is a key to unify the symplectic and complex structures \cite{Hit}. 
Such structures have been extensively studied in 2-dimensions by Huybrechts \cite{Huy}. 
In this section, for the sake of completeness, we provide a brief summary of Huybrechts' work.

\subsection{Generalized Calabi--Yau structures}
Let $M$ be a differentiable manifold underlying a K3 surface and $A^{2*}_\C(M)=\oplus_{i=0}^2 A_\C^{2i}(M)$ the space of even differential forms with $\C$-coefficients. 
Let $\varphi_i$ denote the degree $i$ part of $\varphi \in A^{2*}_\C(M)$. 
We define a pairing on $A_\C^{2*}(M)$ by
$$
\langle \varphi, \varphi' \rangle := \varphi_2 \wedge \varphi'_2 -  \varphi_0 \wedge \varphi'_4 -  \varphi_4 \wedge \varphi'_0 \in A_\C^{4}(M), 
$$
which is the Mukai pairing on the level of differential forms.

\begin{Def}[generalized Calabi--Yau structure]
A generalized Calabi--Yau structure on $M$ is a closed form $\varphi \in A^{2*}_\C(M)$ such that
$$
\langle \varphi, \varphi \rangle =0, \ \ \ \langle \varphi, \overline{\varphi} \rangle >0. 
$$
\end{Def}

The special appeal of generalized Calabi--Yau structures resides in the fact that they embrace complex and symplectic structures on an equal footing. 
There are two fundamental Calabi--Yau structures: 
\begin{enumerate}
\item A symplectic structure $\omega$ on $M$ induces a generalized Calabi--Yau structure 
$$
\varphi=e^{\sqrt{-1}\omega} :=1 + \sqrt{-1} \omega-\frac{1}{2}\omega^2.
$$ 
\item A complex structure $J$ of $M$ makes $M$ a K3 surface $M_J$. 
A holomorphic $2$-form $\sigma$ of $M_J$, which is unique up to scaling, defines a generalized Calabi--Yau structure $\varphi=\sigma$. 
We also write $M_\sigma=M_J$. 
\end{enumerate}

For $B \in A^{2}_\C(M)$, $e^{B}$ acts on $A^{2*}_\C(M)$ by exterior product, i.e. 
$$
e^{B}\varphi :=(1 + B + \frac{1}{2}B \wedge B) \wedge \varphi.
$$ 
This action is orthogonal with respect to the pairing $\langle-,-\rangle$, namely 
$$
\langle e^{B} \varphi,  e^{B}\varphi' \rangle  = \langle \varphi, \varphi'  \rangle.
$$ 
A real closed 2-form is called a $B$-field. 
For a $B$-field $B$ and a generalized Calabi--Yau structure $\varphi$, 
the $B$-field shift $e^{B} \varphi$ is a generalized Calabi--Yau structure.

We will later show that a $B$-field is indispensable when we view complex and symplectic structures as special instance of a more general notion (Remark \ref{connect type A and B}).  
The following shows that a generalized Calabi--Yau structure is a $B$-field shift of either one of the fundamental Calabi--Yau structures. 

\begin{Prop}[\cite{Hit}] \label{type A and B}
Let $\varphi$ be a generalized Calabi--Yau structure. 
\begin{enumerate}
\item[(A)] If $\varphi_0 \ne0$, then
$$
\varphi=e^B(\varphi_0 e^{\sqrt{-1}\omega})=\varphi_0 e^{B+\sqrt{-1}\omega}
$$
with a symplectic $\omega$ and a $B$-field $B$. 
\item[(B)] If $\varphi_0=0$, then 
$$
\varphi=e^{B}\sigma= \sigma + B^{0,2} \wedge  \sigma
$$
with a holomorphic 2-form $\sigma$ (for a complex structure) and a $B$-field $B$. 
$B^{0,2}$ is the $(0,2)$-part of $B$ with respect to $\sigma$. 
\end{enumerate}
\end{Prop}

Generalized Calabi--Yau structures of (A) and (B) in Proposition \ref{type A and B} are called type $A$ and type $B$ respectively. 

We consider the group $\Diff_*(M)$ of the diffeomorphisms $f$ of $M$ such that the induced action $f^*$ on $H^2(M,\Z)$ is trivial. 
Generalized Calabi--Yau structures $\varphi$ and $\varphi'$ are called isomorphic if there exist an exact $B$-field $B$ and a diffeomorphism $f \in \Diff_*(M)$ 
such that $\varphi=e^{B} f^*\varphi'$.

The most fascinating aspect of generalized Calabi--Yau structures is the occurrence of the classical Calabi--Yau structure $\sigma$ (type $B$)
as well as of symplectic generalized Calabi--Yau structure $e^{\sqrt{-1}\omega}$ (type $A$) in the same moduli space. 
This allows us to pass from the symplectic to the complex world in a continuous fashion. 

\begin{Rem}[\cite{Hit}] \label{connect type A and B}
For a complex structure $\sigma$, the real and imaginary parts $\re(\sigma), \im(\sigma)$ are symplectic forms.  
A family of generalized Calabi--Yau structures of type $A$
$$
\varphi_t:=t e^{\frac{1}{t}(\re(\sigma) + \sqrt{-1} \im(\sigma))}=t+\sigma
$$
converges, as $t$ goes to $0$, to the generalized Calabi--Yau structure $\sigma$ of type $B$. 
In this way, the $B$-fields interpolate between generalized Calabi--Yau structures of type $A$ and $B$.  
\end{Rem}


\subsection{(hyper)K\"ahler structures} \label{hKahler str}

\begin{Def}
Let $\varphi$ be a generalized Calabi--Yau structure.  
We define $P_\varphi \subset A^*(M)$ and $P_{[\varphi]} \subset H^*(M,\R)$ by 
\begin{align}
P_\varphi & := \R \mathrm{Re} \varphi \oplus  \R \mathrm{Im} \varphi \subset  A^*(M), \notag \\
P_{[\varphi]} & := \R [\mathrm{Re} \varphi] \oplus  \R [\mathrm{Im} \varphi] \subset  H^*(M,\R) \notag 
\end{align}
i.e. the $\R$-vector spaces spanned by the real and imaginary parts of $\varphi$ and $[\varphi]$ respectively. 
\end{Def}
$P_\varphi$ is positive with respect to the pairing $\langle-,-\rangle$ at every point.  
$P_{[\varphi]} \subset H^*(M,\R) \cong \R^{4,20}$ is a positive 2-plane with respect to the Makai pairing. 

We say that generalized Calabi--Yau structures $\varphi$ and $\varphi'$ are orthogonal if $P_\varphi $ and $P_{\varphi'}$ are pointwise orthogonal. 
This orthogonality is a stronger condition than just $\langle \varphi, \varphi' \rangle=0$.

\begin{Def}[K\"ahler]
A generalized Calabi--Yau structure $\varphi$ is called K\"ahler if there exists another generalized Calabi--Yau structure $\varphi'$ orthogonal to $\varphi$. 
In this case, $\varphi'$ is called a K\"ahler structure for $\varphi$. 
\end{Def}

\begin{Ex}
Let $\varphi=\sigma$ be a holomorphic $2$-form for a complex structure on $M_\sigma$. 
If $\varphi'$ is a K\"ahler structure for $\varphi$, then it is of the form $\varphi'=\varphi'_0e^{B+\sqrt{-1}\omega}$ (type $A$) 
for some scalar $\varphi_0'$, $B$-field $B$ and symplectic form $\omega$.  
The orthogonality is equivalent to $\sigma \wedge B = \sigma \wedge \omega =0$. 
Therefore $B$ is a closed real $(1,1)$-form and either $\pm \omega$ is a K\"ahler form with respect to $\sigma$. 
\end{Ex}

A hyperK\"ahler structure is then defined as a special instance of a K\"ahler structure. 
Recall first that a K\"ahler form $\omega$ on a K3 surface is a hyperK\"ahler form if $2\omega^2=C\sigma \wedge \overline{\sigma}$ for some $C \in \R$, 
which may be assumed to be $1$ by rescaling $\sigma$.

\begin{Def}[hyperK\"ahler]
A generalized Calabi--Yau structure $\varphi$ is hyperK\"ahler if there exists a K\"ahler structure $\varphi'$ for $\varphi$ 
such that 
$$
\langle \varphi, \overline{\varphi} \rangle = \langle \varphi', \overline{\varphi'} \rangle.
$$ 
Such $\varphi'$ is called a hyperK\"ahler structure for $\varphi$. 
\end{Def}

\begin{Rem} \label{Rem B-shifts}
\begin{enumerate}
\item The definition of a (hyper)K\"ahler structure is symmetric for $\varphi$ and $\varphi'$. 
\item Let $\varphi$ be a generalized Calabi--Yau structure and $\varphi'$ a (hyper)K\"ahler for $\varphi$. 
Then $e^{B} \varphi'$ is a (hyper)K\"ahler structure for $e^{B} \varphi$. 
\end{enumerate}
\end{Rem}

For the later use, we give a classification of the hyperK\"ahler structures for a generalized Calabi--Yau structure $\varphi$. 
By Remark \ref{Rem B-shifts}, we may first assume that $\varphi$ is either of the form $\varphi=\lambda e^{\sqrt{-1}\omega}$ or $\varphi=\sigma$.   
\begin{enumerate}
\item[(A)] 
If $\varphi=\lambda e^{\sqrt{-1}\omega}$, then a hyperK\"ahler structure is either 
\begin{enumerate}
\item[(A)] a generalized Calabi--Yau structure $\varphi'=\lambda' e^{B'+\sqrt{-1}\omega'}$ such that
\begin{enumerate}
\item $\omega \wedge \omega'=\omega \wedge B'=\omega' \wedge B = 0$, $B'^2=\omega^2+\omega'^2$, 
\item $|\lambda|^2 \omega^2 = |\lambda'|^2 \omega'^2$, 
\end{enumerate}
\item[(B)] a generalized Calabi--Yau structure $\varphi'=\sigma$, where either $\pm \omega$ is a hyperK\"ahler form such that $2|\lambda|^2\omega^2=\sigma \wedge \overline{\sigma}$.
\end{enumerate}
\item[(B)] 
If $\varphi=\sigma$, then a hyperK\"ahler structure is a generalized Calabi--Yau structure $\varphi'=\lambda e^{B+\sqrt{-1}\omega}$,  
where $B$ is a closed $(1,1)$-form and either $\pm \omega$ is a hyperK\"ahler form such that $2|\lambda|^2\omega^2=\sigma \wedge \overline{\sigma}$. 
\end{enumerate}
Any hyperK\"ahler structure is a $B$-field shift of one of the above cases. 
In summary, there are only 3 possible combinations: 
\begin{center}
(type $A$, type $B$), (type $B$, type $A$), (type $A$, type $A$).
\end{center}


\subsection{Generalized K3 surfaces} 

\begin{Def}[generalized K3 surface] 
A generalized K3 surface is a pair $(\varphi,\varphi')$ of generalized Calabi--Yau structures 
such that $\varphi$ is a hyperK\"ahler structure for $\varphi'$. 
\end{Def}

\begin{Rem}
A K3 surface $M_\sigma$ with a chosen hyperK\"ahler structure $\omega$ (which means $2 \omega^2=\sigma \wedge \overline{\sigma}$)  
can be identified with a generalized K3 surface $(e^{\sqrt{-1}\omega},\sigma)$.  
\end{Rem}

A $B$-field shift $e^{B} (\varphi,\varphi')=(e^{B}\varphi,e^{B}\varphi')$ of a generalized K3 surface $(\varphi,\varphi')$ is a generalized K3 surface. 
Generalized K3 surfaces $(\varphi,\varphi')$ and $(\psi,\psi')$ are called isomorphic if there exist a diffeomorphism $f \in \Diff_*(M)$ and an exact $B$-field $B \in A^2(M)$ 
such that $(\varphi,\varphi')=e^{B} f^*(\psi,\psi')$. 
As indicated in the introduction, the moduli space $\mathfrak{M}_{\gKt}$ of generalized K3 surfaces has a nice description in terms of the cohomology groups.  
\begin{Thm} \label{Torelli gK3}
The period map 
$$
\mathfrak{per}_{\gKt}: \mathfrak{M}_{\gKt} \rightarrow \Gr^{po}_{2,2}(H^*(M,\R)), \ \ \ (\varphi,\varphi')\mapsto (P_{[\varphi]},P_{[\varphi']})
$$
is an immersion with dense image. 
Here $\Gr^{po}_{2,2}(H^*(M,\R))$ is the Grassmannian parametrizing the orthogonal pairs of positive oriented 2-planes in $H^*(M,\R) \cong \R^{4,20}$. 
\end{Thm}
A point in the complement of the image of $\mathfrak{per}_{\gKt}$ is orthogonal to a $(-2)$-class and can be interpreted as a degenerate structure. 


\subsection{Moduli spaces and period domains}

We lastly compare the moduli space $\mathfrak{N}_{\mathrm{gCY}} :=\{\C \varphi\}/\cong$ of the generalized Calabi--Yau structures of hyperK\"ahler type, 
and the moduli space $\mathfrak{N}_{\mathrm{K3}} :=\{\C \sigma\}/\mathrm{Diff}_*(M)$ of the classical marked K3 surfaces. (Figure \ref{Moduli pic})

\begin{Thm} \label{Torelli theorem}

The classical period map $\mathfrak{per}_{\Kt}$ naturally extends to the period map $\mathfrak{per}_{\mathrm{gCY}}$ for the generalized Calabi--Yau structures of hyperK\"ahler type: 
$$
 \xymatrix@=18pt{
  \mathfrak{N}_{\mathrm{gCY}} \ar@{}[d]|{\bigcup}  \ar[rr]^-{\mathfrak{per}_{\mathrm{gCY}}}_-{\C \varphi \to [\varphi]}  \ \ 
  & &\ar@{}[d]|{\bigcup}  \widetilde{\DD}:=\{ [\varphi] \in \mathbb{P}(H^*(M,\C)) \ | \ \langle \varphi, \varphi \rangle=0, \langle \varphi, \overline{\varphi}  \rangle>0 \}  \\
 \mathfrak{N}_{\Kt}  \ar[rr]^-{\mathfrak{per}_{\Kt}}_-{\C \sigma \to [\sigma]} \ \  
 & &  \DD :=\{ [\sigma] \in \mathbb{P}(H^2(M,\C)) \ | \ \langle \sigma, \sigma \rangle=0, \langle \sigma, \overline{\sigma} \rangle >0 \}  
}
$$
$\mathfrak{per}_{\mathrm{gCY}}$ is \'etale surjective, and bijective over the complement of the hyperplane section $\mathbb{P}(H^2(M,\C)\oplus H^4(M,\C))\cap \widetilde{\DD}$. 
\end{Thm}

Therefore the generalized Calabi--Yau structures of hyperK\"ahler type may be considered as geometric realizations of the points in the extended period domain $ \widetilde{\DD}$.  
For this reason, we hereafter regard a generalized Calabi--Yau structure of hyperK\"ahler type as an element in $H^*(M,\C)$. 

\begin{figure}[htbp]
 \begin{center} 
\includegraphics[width=70mm]{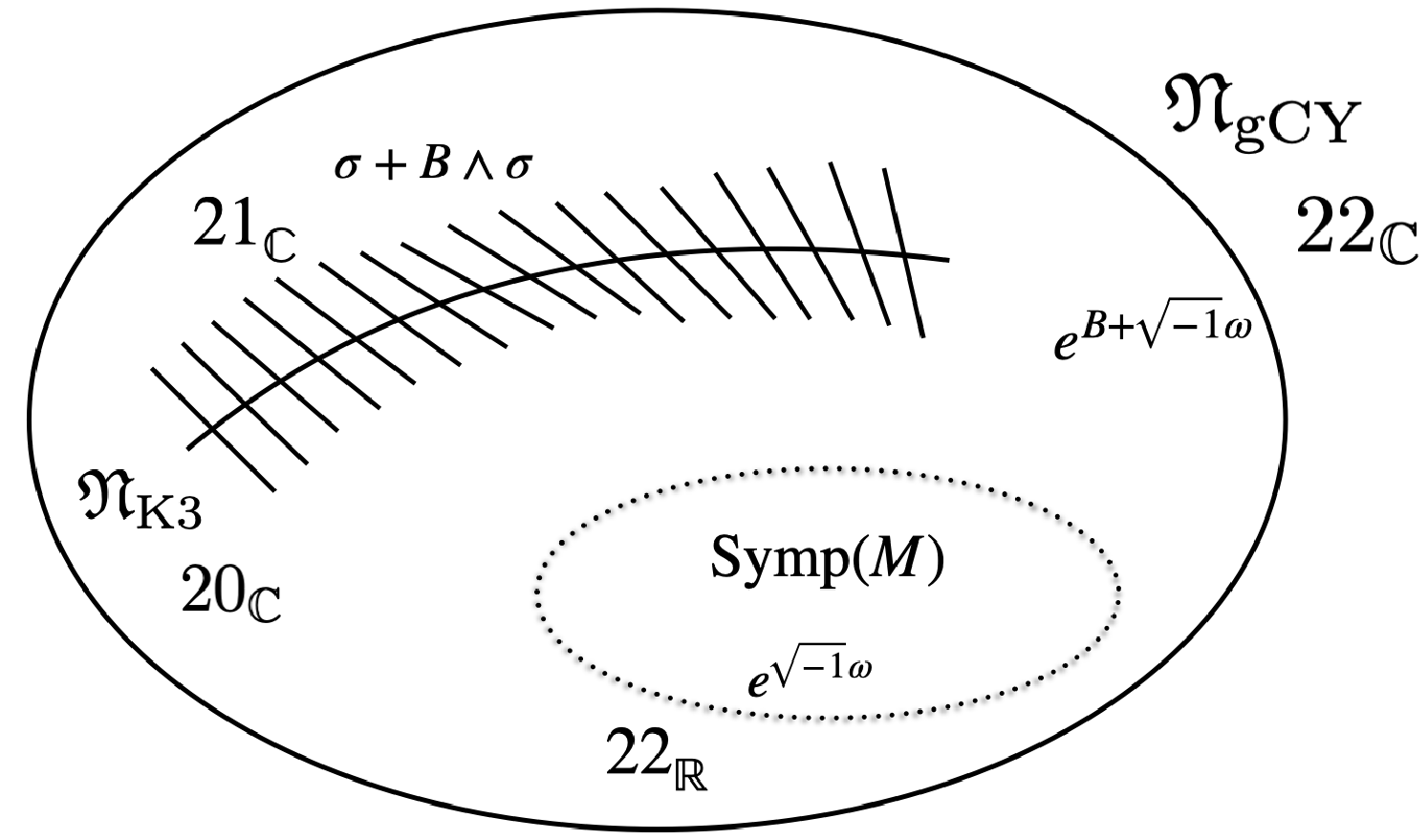}
\caption{Moduli space  $\mathfrak{N}_{\mathrm{gCY}}$} \label{Moduli pic}
 \end{center} 
\end{figure}


\section{Rigid structures}  \label{rigid structures}


\subsection{Lattices and K3 surfaces}

In order to fix notation, let us begin with some basics of lattices and K3 surfaces. 
The hyperbolic lattice $U$ is the rank $2$ even unimodular lattice defined by the Gram matrix $\begin{bmatrix} 0 & 1\\ 1 & 0\\ \end{bmatrix}$. 
$E_{8}$ is the rank $8$ negative-definite even unimodular lattice defined by the corresponding Cartan matrix. 
We denote by $\Lambda_{K3}:=U^{\oplus 3}\oplus E_8^{\oplus 2}$ the K3 lattice. 
For a lattice $L$, we write $L_\Q:=L \otimes_\Z \Q$, and $L_\R$ and $L_\C$ are defined in a similar manner. 

Let $S$ be a K3 surface. 
Then $H^2(S,\Z)$ with its intersection form is isomorphic to the K3 lattice $\Lambda_{K3}$. 
The N\'eron-Severi lattice and transcendental lattice are defined respectively by
$$
NS(S):=H^{1,1}(S,\R) \cap H^2(S,\Z), \ \ \ T(S):=NS(S)^\perp, 
$$
where the orthogonal complement is taken in $H^2(S,\Z)$. 
The Picard number is $\rho(S):=\rank (NS(S))$.  
The extended N\'eron-Severi lattice $NS'(S)$ is the sublattice 
\begin{align*}
NS'(S) :&= H^0(S,\Z) \oplus NS(S) \oplus H^4(S,\Z) \\
& \cong NS(S) \oplus U
\end{align*}
of the Mukai lattice $H^*(S,\Z) \cong U^{\oplus 4}\oplus E_8^{\oplus 2}$.


\subsection{N\'eron--Severi lattice and transcendental lattice} \label{NS and T}

We introduce the N\'eron--Severi lattice and the transcendental lattice of a generalized K3 surface $X=(\varphi,\varphi')$.  

\begin{Def}  \label{NS and T def}
The N\'eron--Severi lattice and transcendental lattices of $X=(\varphi,\varphi')$ are defined respectively by 
\begin{align*}
\widetilde{NS}(X) :=\{ \delta \in H^*(M,\Z) \ | \ \langle \delta, \varphi' \rangle =0 \}, \\
\widetilde{T}(X) :=\{ \delta \in H^*(M,\Z) \ | \ \langle \delta, \varphi \rangle =0 \}. 
\end{align*}
\end{Def}

In \cite{Huy} the Picard group $\mathrm{Pic}(\varphi)$ and the transcendental lattice $T(\varphi)$ are defined for a generalized Calabi--Yau structure $\varphi$.  
Our definition of the transcendental lattice is different from \cite{Huy}, where it is defined as the orthogonal complement of the Picard group. 
 
It is important that we define $\widetilde{NS}(X)$ and $\widetilde{T}(X)$ on a completely equal footing, 
and hence the intersection $\widetilde{NS}(X) \cap \widetilde{T}(X)$ may be non-trivial. 
Moreover, it is easy to see that 
$$
\varphi \in \widetilde{NS}(X)_\C, \ \ \ \varphi' \in \widetilde{T}(X)_\C,
$$
and 
$$
2 \le \rank(\widetilde{NS}(X)) \le 22, \ \ \ 
2 \le \rank(\widetilde{T}(X)) \le 22.
$$
For example, $\varphi$ is chosen generically in $\widetilde{NS}(X)_\C$, then $\widetilde{NS}(X)^\perp=\widetilde{T}(X)$. 

For later use, given $\psi \in H^*(M,\C)$, we denote by $L_\psi$ the smallest sublattice $L \subset H^*(M,\Z)$ such that $\psi \in L_\C$.   
Then we may write 
$$
\widetilde{NS}(X)=L_{\varphi'}^\perp, \ \ \ \widetilde{T}(X)=L_{\varphi}^\perp.
$$

\begin{Rem} \label{no longer algebraic}
In general, $H^0(M,\Z), H^4(M,\Z) \not \subset  \widetilde{NS}(X)$.     
This feature has an important consequence in our formulation of mirror symmetry (Section \ref{MS for gK3}).  
\end{Rem}

The following examples will play an important role in Section \ref{MS for gK3}. (cf. Example 4.4 of \cite{Huy} )

\begin{Ex}\label{NS and T example}
We denote by $\delta_i$ the degree $i$ part of $\delta \in H^*(M,\Z)$. 
\begin{enumerate}
\item 
If $\varphi'=\sigma + B^{0,2} \wedge \sigma$, a $B$-field shift of $\sigma$ by $B \in H^2(M,\R)$, then 
\begin{align*}
\widetilde{NS}(X) &= \{\delta_0+ \delta_2 \ | \ \int_M \delta_2 \wedge \sigma = \delta_0 \int_M B^{0,2} \wedge \sigma \} \oplus  H^4(M,\Z) \\
& = e^B NS'(M_\sigma)_\R \cap H^*(M,\Z). 
\end{align*}
There is a natural inclusion $NS(M_\sigma) \oplus H^4(M,\Z) \subset \widetilde{NS}(X)$. 
The inclusion is strict if $B\in H^2(M,\Q)$, and then $\widetilde{NS}(X)$ contains the lattice 
$$
\big(\Z(1,B,\frac{1}{2}B^2) \cap H^*(M,\Z) \big) \oplus H^4(M,\Z) \cong U(m)
$$ 
for $m$ minimal with $m(1,B,\frac{1}{2}B^2) \in H^*(M,\Z)$. 
In particular, for the classical case $\varphi'=\sigma$, we recover 
\begin{align*}
\widetilde{NS}(X) &= H^0(M,\Z) \oplus NS(M_\sigma)\oplus H^4(M,\Z) \\
& = NS'(M_\sigma). 
\end{align*}

\item  If $\varphi=e^{B+\sqrt{-1}\omega}$ for $B \in H^2(M,\R)$ and a symplectic form $\omega$, the condition $\langle \delta,\varphi \rangle=0$ is equivalent to
\begin{align*}
\int_M \delta_2 \wedge B - \frac{\delta_0}{2}\int_M(B^2-\omega^2) + \int_M \delta_4 =0, \\
\int_M \delta_2 \wedge \omega - \delta_0 \int_M B \wedge \omega=0. 
\end{align*} 
In particular, if $\sigma$ is a holomorphic 2-form for a complex structure on $M$, 
then for a generic choice of $B,\omega \in NS(M_\sigma)$ 
$$
\widetilde{T}(X)=T(M_\sigma).
$$
\end{enumerate}

\end{Ex}


\subsection{Complex rigidity}

The singular K3 surfaces, also known as complex rigid K3 surfaces, do not admit any complex deformation while maintaining the maximal Picard number $20$
(or equivalently, the minimal rank $2$ of the transcendental lattice).   
We first provide a generalization of the singular K3 surfaces.

\begin{Def}
A generalized K3 surface $X=(\varphi,\varphi')$ is called complex rigid if $\varphi'$ is of type $B$ and $\rank (\widetilde{NS}(X))=22$. 
\end{Def}

Note that the condition $\rank (\widetilde{NS}(X))=22$ is equivalent to $\rank(L_{\varphi'})=2$. 

\begin{Thm} \label{rigid K3}
A complex rigid generalized K3 surface is of the form $X=(\lambda e^{B+B'+\sqrt{-1}\omega}, \sigma+B'\wedge \sigma)$, where 
\begin{enumerate}
\item $M_\sigma$ is a singular K3 surface, 
\item $B \in H^{1,1}(M_\sigma,\R)$, $B' \in H^2(M,\Q)$,  
\item either $\pm \omega$ is a hyperK\"ahler form for $\sigma$. 
\end{enumerate}
In other words, $X$ is a rational $B$-field shift of a singular K3 surface equipped with a complexified K\"ahler structure.    
\end{Thm}
\begin{proof}
Let $X=(\varphi, \varphi')$ be complex rigid.  
By the classification of hyperK\"ahler structures (comments below Remark \ref{Rem B-shifts}), $\varphi$ must be of type $A$.  
Hence we can write $X=e^{B'}(\lambda e^{B+\sqrt{-1}\omega},\sigma)$, where $Y=(\lambda e^{B+\sqrt{-1}\omega},\sigma)$ is a generalized K3 surface, 
i.e. $B$ is a closed real $(1,1)$-form and either $\pm \omega$ is a hyperK\"ahler form for $\sigma$.  
On the other hand, for $\varphi'=\sigma+B'\wedge \sigma$, $\rank(L_{\varphi'})=2$ if and only if $\rank(T(M_\sigma))=2$ and $B'$ is rational. 
This can be seen as follows. 
If $\rank(L_{\varphi'})=2$, 
there exist $a,b \in \R$ and $u_1,u_2 \in H^*(M,\Z)$ such that 
$$
\varphi' = \sigma+B'\wedge \sigma = a u_1 + \sqrt{-1}b u_2. 
$$
For $i=1,2$, let $v_i=(u_i)_2\in H^2(M,\Z)$ be  the degree 2 part of $u_i$.  
Then we may write $\sigma= a v_1 + \sqrt{-1}b v_2$ and hence $\rank(L_{\sigma})=2$, which means $\rank(T(M_\sigma))=2$. 
Moreover, we have
\begin{align}
\varphi'=e^B \sigma=a (v_1+B'\wedge v_1)+\sqrt{-1}b (v_2+B'\wedge v_2) \label{phi'}, 
\end{align}
and in order for $\rank(L_{\varphi'})=2$, the $B$-field $B'$ must be rational. 
Conversely, assume $\rank(T(M_\sigma))=2$ and $B'$ is rational, 
then there exist $a,b \in \R$ and $v_1,v_2 \in T(M_\sigma)$ such that $\sigma= a v_1 + \sqrt{-1}b v_2$, 
and $\varphi'$ is of the form of the equation (\ref{phi'}). 
Since $B'$ is rational, $v_1+B'\wedge v_1, v_2+B'\wedge v_2 \in H^*(M,\Q)$ and hence $\rank(L_{\varphi'})=2$. 
\end{proof}


\subsection{K\"ahler rigidity} \label{Kahler rigidity computation}
A K\"ahler rigidity comparable with a complex rigidity naturally appears in the framework of generalized Calabi--Yau structures. 

Here is a key observation. 
Let $S$ be a K3 surface with Picard number $1$. 
We write $NS(S)=\Z H$ with $H^2=2n>0$, and consider 
$$
v_1:=(1,0,-n), \ \ v_2:=(0,H,0) \in NS'(S),
$$
where the notation follows the decomposition $NS'(S)=H^0(S,\Z) \oplus NS(S)\oplus H^4(S,\Z)$. 
Then 
\begin{align}
e^{\sqrt{-1}H} &= (1,\sqrt{-1}H,-n) \notag \\
&=v_1+\sqrt{-1}v_2 \in (\Z v_1 + \Z v_2)_\C \subsetneq NS'(S)_\C.  \notag 
\end{align}
Here 
$$
\Z v_1 + \Z v_2 \cong \langle 2n \rangle^{\oplus 2},
$$ 
where $\langle k \rangle$ denotes the lattice of rank $1$ generated by $v$ with $v^2=k$. 
On the other hand, for $\epsilon^2 \notin \Q$, 
\begin{align}
e^{\sqrt{-1}\epsilon H} &=(1,\sqrt{-1} \epsilon H,-\epsilon^2n) \notag \\
&= (1,0,-\epsilon^2n) + \sqrt{-1}\epsilon(0, H, 0) \notag \\
&= (1,0,0) -\epsilon^2 (0,0,n) + \sqrt{-1}\epsilon(0, H, 0) \in NS'(S)_\C. \notag 
\end{align}
Hence there is no proper sublattice $L \subsetneq NS'(S)$ such that $e^{\sqrt{-1}\epsilon H} \in L_\C$.  
Therefore the K\"ahler structure $H$ is not continuously deformable in such a way that $\rank(L_{e^{\sqrt{-1}\epsilon H}})=2$.    
This new sublattice $L_{e^{B+\sqrt{-1}\omega}}$ is comparable with the transcendental lattice, which is the minimal sublattice $L \subset H^2(M_\sigma,\Z)$ such that $\sigma \in L_\C$.

\begin{Def}
A symplectic manifold $(M,\omega)$ is called symplectic rigid if $\omega^2 \in H^4(M,\Q)$. 
\end{Def}

As the previous calculation shows, $(M,\omega)$ is symplectic rigid if and only if $\rank(L_{e^{\sqrt{-1}\omega}})=2$.  

\begin{Def}
A generalized K3 surface $X=(\varphi,\varphi')$ is called K\"ahler rigid if $\varphi$ is of type $A$ and $\rank (\widetilde{T}(X))=22$.
\end{Def}

Again, the condition $\rank (\widetilde{T}(X))=22$ is equivalent to $\rank(L_{\varphi})=2$.  
We will next provide a characterization of the K\"ahler rigidity.  
Recall that, in contrast to the complex rigidity, the hyperK\"ahler structure $\varphi'$ can be either of type $A$ or $B$.  

\begin{Thm} 
A K\"ahler rigid generalized K3 surface is of the form $X=(\lambda e^{B+\sqrt{-1}\omega}, \varphi')$, where $B \in H^2(M,\Q)$ and $\omega^2 \in H^4(M,\Q)$.  
In other words, $X$ is a rational $B$-field $B$ shift of a symplectic rigid manifold equipped with a hyperK\"ahler structure.    
\end{Thm}
\begin{proof}
Let $X=(e^{B+\sqrt{-1}\omega},\varphi')$ be K\"ahler rigid.  
We consider an existence condition of a rank $2$ sublattice $L \subset H^*(M,\Z)$ such that 
$$
e^{B+\sqrt{-1}\omega}=1 + B + \frac{1}{2}(B^2-\omega^2)+\sqrt{-1}(\omega+ B \wedge \omega) \in L_\C. 
$$
First, $B$ needs to be rational, and hence so is $\omega^2$. 
Then we may write the symplectic class $\omega=\kappa H$ for $\kappa^2 \in \Q$ and $H \in H^2(M,\Z)$ with $H^2>0$. 
In fact, in this case, there exist $m,n \in \N$ such that 
$$
m \mathrm{Re} (e^{B+\sqrt{-1}\kappa H}), \ n \mathrm{Im} (e^{B+\sqrt{-1}\kappa H})  \in H^*(M,\Z). 
$$
Then the complexification $L_\C$ of the lattice 
$$
L=\Z m \mathrm{Re} (e^{B+\sqrt{-1}\kappa H}) + \Z n \mathrm{Im} (e^{B+\sqrt{-1}\kappa H}) \subset H^*(M,\Z)
$$
contains $e^{B+\sqrt{-1}\kappa H}$. 
\end{proof}


\section{Mirror symmetry for generalized K3 surfaces} \label{MS for gK3}

Our formulation of mirror symmetry for generalized K3 surfaces builds upon the combination of two novel ideas: 
(1) lattice polarizations of the cycles (Dolgachev \cite{Dol}) and (2) generalized K3 structures (Hitchin \cite{Hit} and Huybrechts \cite{Huy}).


\subsection{Mirror symmetry for K3 surfaces} \label{MS ala Dolgachev}
We begin with a brief summary of mirror symmetry for lattice polarized K3 surfaces originally introduced by Dolgachev in \cite{Dol}. 
Fix an integer $1 \le \rho \le 19$. 

\begin{Def}
Let $K$ be a lattice of signature $(1,\rho-1)$.  
A $K$-polarized K3 surface is a pair $(S,i)$ of a K3 surface $S$ and a primitive lattice embedding $i:K \hookrightarrow NS(S)$ such that $i(K)$ contains an ample class. 
\end{Def}

Two $K$-polarized K3 surfaces $(S,i)$ and $(S',i')$ are called isomorphic if there is an isomorphism $f:S\rightarrow S'$ of K3 surfaces such that $f^*\circ i'=i$. 

\begin{Def} \label{DolDef}
For a lattice $K$ of signature $(1,\rho-1)$, assume there exists a lattice $L$ such that 
$$
K^\perp=L \oplus U.
$$  
Then the family $\mathcal{S}$ of $K$-polarized K3 surfaces and the family $\mathcal{S}^\vee$ of $L$-polarized K3 surfaces are mirror symmetric. 
\end{Def}

In the above setting, $L$ is of signature $(1,19-\rho)$ and uniquely determined with $L^\perp=K\oplus U$ (independent of the chose of a decomposition and an embedding) \cite{Dol}.  
Then, for a generic $K$-polarized K3 surface $S$ and a generic $L$-polarized K3 surface $S^\vee$, we have 
$$
NS'(S) \cong K \oplus U \cong T(S^\vee), \ \ \ 
T(S) \cong L \oplus U \cong NS'(S^\vee). 
$$
They are considered as a duality between the algebraic cycles of $S$ and the transcendental cycles of $S^\vee$ as lattices, and vice versa (duality of Yukawa couplings). 
Mirror symmetry can also be explored at the level of moduli spaces. 
See for example \cite{Dol, FKY, HK2}.

The crucial assumption in Dolgachev's formulation is the existence of a decomposition of the form $K^\perp=L \oplus U$,  
which reflects the decomposition 
$$
NS'(S) = NS(S) \oplus \big(H^0(S,\Z)\oplus H^4(S,\Z)\big)
$$
under mirror symmetry. 
Hence the assumption is indispensable as long as we consider geometry where the point class and fundamental class play special roles. 

\begin{Rem}
It is also worth noting that there are geometric explanations for such a decomposition; the hyperbolic lattice $U$ corresponds to: 
\begin{enumerate}
\item a cusp (large complex structure limit) of the Baily--Borel compactification of the period domains \cite{Sca},  
\item the lattice spanned by the fiber and section classes of a special Lagrangian fibration in SYZ mirror symmetry. 
\end{enumerate}
Therefore, a choice of such $U$ corresponds to a choice of a mirror K3 surface. 
\end{Rem}


\subsection{Problems} \label{MS problem}

Although Dolgachev's formulation works well in many cases, it has a few issues and cannot be considered definitive. 
First of all, $NS'(S)$ and $T(S)$ are not truly symmetric; 
the minimal rank of $NS'(S)$ is $3$ (as long as $S$ is algebraic) while the minimal rank of $T(S)$ is $2$. 
Additionally, the assumption $K^\perp=L \oplus U$ does not hold in general. 
These issues fall into two categories. 
\begin{enumerate}
\item There are decompositions of the form $K^\perp=L \oplus U(m)$ but for $m \ge 2$. 
Examples of such K3 surfaces are constructed for examples in \cite{HK}. 
The duality between the algebraic cycles of a $K$-polarized K3 surface and the transcendental cycles of an $L$-polarized K3 surface holds only over $\Q$.   
\item The situation is worse if $K^\perp$ is of signature $(2,0)$ and hence cannot be of the form $U(m)$ for any $m \ge 1$. 
Such a lattice $K$ is given by $K=NS(S)$ for a singular K3 surface $S$. 
Therefore, the natural question ``How can we understand mirror symmetry for singular K3 surfaces?" has been a mysterious problem for a long time. 
Indeed, for a singular K3 surface, the K\"ahler moduli space is of dimension 20 and complex moduli space is of dimension 0, and there seems no mirror partner as a K3 surface.   
\begin{table}[htb]
\begin{center}
  \begin{tabular}{c|c|c}
    & singular K3 surface&   \ \ \  \ \ \ \ \ \ mirror ?? \ \ \  \ \ \ \ \ \   \\ \hline 
\ K\"ahler \ & 20-dim  &  0-dim \\ \hline
\ complex \ & 0-dim &  20-dim  
  \end{tabular}
  \end{center}
\end{table}  
\end{enumerate}

We will settle these problems in the framework of generalized K3 surfaces in Sections \ref{MS_B} and \ref{MS for sing K3}. 

More generally, the existence of rigid Calabi--Yau manifolds poses a serious challenge to mirror symmetry. 
Various approaches have been proposed to tackle this issue, but there seem currently no universally agreed-upon formulation. 
See for example \cite{BB, CDP, Leu, SS}.    


\subsection{Mirror symmetry for generalized K3 surfaces}

We will first introduce the notion of lattice polarizations of the Mukai lattice. 
For integers $\kappa,\lambda \ge 2$ such that $\kappa+\lambda=24$, 
let $K$ and $L$ be even lattices of signature $(2,\kappa-2)$ and $(2,\lambda-2)$ respectively. 

\begin{Def} \label{Mukai lattice pol}
A $(K, L)$-polarized generalized K3 surface is a pair $(X,i)$ of a generalized K3 surface $X=(\varphi,\varphi')$ 
and a lattice embedding $i: K \oplus L \hookrightarrow H^*(M,\Z)$
such that 
\begin{enumerate}
\item the restrictions $i|_K$ and $i|_L$ are primitive embeddings, 
\item $i(K) \subset \widetilde{NS}(X)$ and $i(L) \subset \widetilde{T}(X)$. 
\end{enumerate}
\end{Def}

Two $(K, L)$-polarized generalized K3 surfaces $(X,i)$ and $(Y,j)$ are called isomorphic 
if there is a diffeomorphism $f \in \Diff_*(M)$ such that $X=f^*Y$ and $f^*\circ j=i$. 
For simplicity, we usually omit $i$ from the notation $(X,i)$ when there is no confusion.


\begin{Def} \label{MS for (K,L)-pol gK3s}
A family $\mathcal{X}$ of $(K,L)$-polarized generalized K3 surfaces and a family $\mathcal{Y}$ of $(L, K)$-polarized generalized K3 surfaces are mirror symmetric.
\end{Def}

As mentioned in Remark \ref{no longer algebraic}, in general, the elements in $H^0(M,\Z)$ and $H^4(M,\Z)$ are no longer ``algebraic" 
and there is no natural summand $U \cong H^0(M,\Z) \oplus H^4(M,\Z)$ in $\widetilde{NS}(X)$. 
Therefore, there is no assumption on the lattices $K$ and $L$ in our formulation. 

There may be several families of $(K,L)$-polarized generalized K3 surfaces, and we typically select one for the purpose of mirror symmetry.
This remark is particularly important due to the symmetry of Definition \ref{Mukai lattice pol}. 
Specifically, if $(\varphi,\varphi')$ is a $(K, L)$-polarized generalized K3 surface, then $(\varphi',\varphi)$ is a $(L,K)$-polarized generalized K3 surface in a natural way.  
Nevertheless, as elucidated in Section \ref{MS_B} and Section \ref{MS for sing K3}, we may obtain a non-trivial result by selecting appropriate families. 
This issue is also related to the so-called multiple mirrors phenomenon. 

\begin{Rem} \label{AM SCFT}
Our formulation is compatible with Aspinwall--Morrison's description of mirror symmetry for K3 surfaces from the viewpoint of SCFTs \cite{AM}.  
They  identified the moduli space of $N=(2,2)$ SCFTs on a K3 surface with 
the Grassmannian
$$
\mathrm{Gr}_{2,2}^{po}(H^*(M,\R)) \cong \mathrm{O}(4,20)/(\mathrm{SO}(2)\times \mathrm{SO}(2)\times \mathrm{O}(20)) 
$$
parametrizing the orthogonal pairs of positive oriented 2-planes in $H^*(M,\R) \cong \R^{4,20}$.    
To a K3 surface $M_\sigma$ with a K\"ahler class $\omega$, we may associate 
$$
(P_\sigma, P_{e^{\sqrt{-1}\omega}}) \in \mathrm{Gr}_{2,2}^{po}(H^*(M,\R)). 
$$ 
There are ``non-geometric points" that cannot be obtained in this way. 
However, we now understand that these points are given by generalized K3 surfaces, as stated in Theorem \ref{Torelli gK3}. 

Then mirror symmetry is an involution of $\mathrm{Gr}_{2,2}^{po}(H^*(M,\R))$ given by swapping the two 2-planes, which does not preserve the geometric points. 
Therefore, it is understood that a mirror partner of a K3 surface is not necessarily a K3 surface, but rather a generalized  K3 surface. 
It is also explained that choosing $U$ in $H^*(M,\Z)$ amounts to selecting a specific $N=(2,2)$ SCFT to represent an $N=(4,4)$ SCFT. 
\end{Rem}

\begin{Def}
Let  $X=(\varphi,\varphi')$ be a $(K, L)$-polarized generalized K3 surface. 
A deformation of $\C \varphi$ keeping $X$ as a $(K, L)$-polarized generalized K3 surface is called an $A$-deformation.    
The $A$-moduli space $\mathfrak{M}_A$ is defined as the space of $A$-deformations of $X$.   
A $B$-deformation and the $B$-moduli space $\mathfrak{M}_B$ are defined similarly. 
\end{Def}

We will provide an explicit description of the moduli space $\mathfrak{M}_A \times \mathfrak{M}_B$ of the $(K, L)$-polarized generalized K3 surfaces. 
We begin by introducing
$$
\DD_K :=\{ \varphi \in \mathbb{P}(K_\C) \ | \ \langle \varphi,\varphi \rangle=0, \langle \varphi, \overline{\varphi} \rangle >0 \}, 
$$
a complex manifold of dimension $\kappa-2$, isomorphic to the Grassmannian $\Gr^{po}_{2}(K_\R) \cong \OO(2,\kappa-2)/(\SO(2)\times \OO(\kappa-2))$ 
parametrizing the positive oriented 2-planes in $K_\R \cong \R^{2,\kappa-2}$.  
It has two connected components, each of which is isomorphic to a symmetric domain of type IV. 
$\DD_L$ is defined in a similar manner.

\begin{Thm} \label{moduli and dim theorem}
The moduli space $\mathfrak{M}_A \times \mathfrak{M}_B$ of the $(K, L)$-polarized generalized K3 surfaces 
is identified with an open dense subset of $\DD_{(K,L)}:= \DD_K \times \DD_L$.  
Moreover, $\mathfrak{M}_A$ is identified with an open dense subset of $\DD_K$, and similarly, $\mathfrak{M}_B$ is identified with an open dense subset of $\DD_K$.  
\end{Thm}

\begin{proof}
Let $(X=(\varphi,\varphi'),i)$ be a $(K,L)$-polarized generalized K3 surface. 
Then the conditions $i(K) \subset \widetilde{NS}(X)$ and $i(L) \subset \widetilde{T}(X)$ are equivalent to $\varphi' \in \DD_{i(L)}$ and $\varphi \in \DD_{i(K)}$, respectively. 
We will demonstrate that a point in an open dense subset of $\DD_{i(K)} \times \DD_{i(L)}$ defines a unique generalized K3 surface.
Firstly, we associate to $(\alpha,\beta) \in \DD_{i(K)} \times \DD_{i(L)}$ an orthogonal pair of positive oriented 2-planes $(P_{\alpha},P_{\beta}) \in \Gr^{po}_{2,2}(H^*(M,\R))$,  
which are generically not orthogonal to any $(-2)$-class. 
Therefore, by Theorem \ref{Torelli gK3}, the period map $\mathfrak{per}_{\gKt}: \mathfrak{M}_{\gKt} \rightarrow \Gr^{po}_{2,2}(H^*(M,\R))$ ensures that, 
generically, $(\alpha,\beta) \in \DD_{i(K)} \times \DD_{i(L)}$ defines a unique generalized K3 surface. 
Consequently, the moduli space $\mathfrak{M}_A \times \mathfrak{M}_B$ may be identified with an open dense subset of $\DD_{i(K)} \times  \DD_{i(L)} \cong \DD_{(K,L)}$. 

Let $H:=(H^2(M,\C)\oplus H^4(M,\Z))\cap i(K)$.  
Then, by Theorem \ref{Torelli theorem} and the above discussion, $\DD_{i(K)} \setminus \PP(H)$ parametrizes 
the $A$-deformation of $(\varphi,\varphi')$ for a generic choice of $\varphi' \in \DD_{i(L)}$. 
Therefore, $\mathfrak{M}_A$ is identified with $\DD_{i(K)} \setminus \PP(H)$. 
A parallel discussion can be carried out for $\mathfrak{M}_B$. 
\end{proof}


\subsection{Comparison: classical and new} \label{classical and new}
Next, we will demonstrate that our formulation of mirror symmetry inherently encompasses Dolgachev's formulation. 

For an integer $1 \le \rho \le 19$, let $K'$ be a lattice of signature $(1,\rho-1)$. 
Assume that there exists a decomposition $K'^\perp=L' \oplus U$ for a lattice $L'$.  
Mirror symmetry for lattice polarized K3 surfaces posits that the $K'$-polarized K3 surfaces and the $L'$-polarized K3 surfaces are mirror symmetric to each other.  

Let $M_\sigma$ be a generic $K'$-polarized K3 surface equipped with a $B$-field $B$ shift of a hyperK\"ahler form $\omega$ 
($B, \omega \in NS(M_\sigma)_\R$ are also taken generically). 
It defines a generalized K3 surface $X=(\varphi,\varphi')=(e^{B+\sqrt{-1}\omega}, \sigma)$.  
By the calculation in Example \ref{NS and T example}, 
\begin{align*}
\widetilde{NS}(X) = NS'(M_\sigma)  \cong K' \oplus U, \ \ \ 
\widetilde{T}(X) = T(M_\sigma) \cong L' \oplus U. 
\end{align*}
Similarly, let $M_{\sigma^\vee}$ be a generic $L'$-polarized K3 surface equipped with a $B$-field $B^\vee$ shift of a hyperK\"ahler form $\omega^\vee$. 
It defines a generalized K3 surface $X^\vee=(\varphi,\varphi')=(e^{B^\vee+\sqrt{-1}\omega^\vee}, \sigma^\vee)$ and 
\begin{align*}
\widetilde{NS}(X^\vee) = NS'(M_{\sigma^\vee}) \cong L' \oplus U, \ \ \ 
\widetilde{T}(X^\vee) = T(M_{\sigma^\vee}) \cong K' \oplus U. 
\end{align*}
By setting $K:=K'\oplus U$ and $L:=L'\oplus U$, 
we may think that $X$ is a $(K,L)$-polarized generalized K3 surface and $X^\vee$ is an $(L,K)$-polarized generalized K3 surface. 
Therefore the formulation of mirror symmetry for lattice polarized K3 surfaces is naturally regarded as a special case of ours.


\subsection{Mirror symmetry and $B$-field twists}  \label{MS_B}
Our formulation of mirror symmetry comes into its own when the conventional approach falls short. 
Let us delve into the first issue outlined in Section \ref{MS problem}. 

For an integer $1 \le \rho \le 19$, let $K'$ be a lattice of signature $(1,\rho-1)$. 
Assume that there exists a decomposition $K'^\perp=L' \oplus U(m)$ for a lattice $L'$ and $m \ge 2$.   
We define $K:=K'\oplus U(m)$ and $L:=L'\oplus U(m)$. 

A $(K,L)$-polarized generalized K3 surface can be constructed as follows. 
First we consider $K'$-polarized K3 surface $M_{\sigma}$ so that generically $T(M_{\sigma})\cong L' \oplus U(m)$. 
Let $e,f$ be the standard basis of this $U(m)$, i.e. $e^2=f^2=0$ and $\langle e,f \rangle =m$, and consider a $B$-field $B'=\frac{1}{m}e \in T(M_{\sigma})_\Q$. 
Then, for $B, \omega \in NS(M_\sigma)_\R$, 
$$
X=e^{B'}(e^{B+\sqrt{-1}\omega}, \sigma)
$$
is a $(K,L)$-polarized generalized K3 surface. 
Indeed, for generic $X$, by the calculation in Example \ref{NS and T example}, 
\begin{align*}
\widetilde{NS}(X) & = NS(M_\sigma) \oplus \Z(m,e,0)\oplus H^4(M_\sigma,\Z) \\
& \cong  K' \oplus U(m)
\end{align*}
and by the choice of $B'$, 
\begin{align*}
\widetilde{T}(X) & = e^{B'}T(M_\sigma)_\R \cap H^*(M_\sigma,\Z) \\
& = L' \oplus \Z(0,e,0) \oplus \Z (0,f,1) \\
& \cong L' \oplus U(m). 
\end{align*}
Note that the elements in $H^0(M,\Z)$ are no longer algebraic in this case. 
An $(L,K)$-polarized K3 surface can be constructed in the same manner. 

\begin{Rem} \label{twisted sheaves}
A $B$-field shift of a complex structure $\sigma$ is thought of as a non-commutative deformation of a K3 surface $M_\sigma$ 
where the geometric objects are twisted by the Brauer group element $\alpha$ associated with the $B$-field. 
The corresponding category is the derived category of twisted coherent sheaves $\mathrm{D}^b(M_\sigma,\alpha)$. 
See for example \cite{Huy}. 
\end{Rem}


\subsection{Mirror symmetry for singular K3 surfaces}  \label{MS for sing K3}
Let us next consider the second problem in Section \ref{MS problem}, 
namely the problem of mirror symmetry for singular K3 surfaces. 

For an integer $n>0$, we define 
$$
K:=\langle -2n \rangle^{\oplus 2} \oplus U^{\oplus 2} \oplus E_8^{\oplus 2}, \ \ \ L:=\langle 2n \rangle^{\oplus 2}. 
$$ 

\begin{enumerate}
\item A family $\mathcal{X}$ of the $(K, L)$-polarized generalized K3 surfaces is given by $\{X=(e^{B +\sqrt{-1}\omega}, \sigma)\}$, 
where $T(M_\sigma)=L$ and $B, \omega \in NS(M_\sigma)_\R$. 
Note that $M_\sigma$ is realized as a $(\langle -2n \rangle^{\oplus 2} \oplus U \oplus E_8^{\oplus 2})$-polarized K3 surface. 
For generic $X$, we have 
$$
K \cong \widetilde{NS}(X)=NS'(M_\sigma), \ \ \ L \cong \widetilde{T}(X)=T(M_\sigma)
$$

\item A family $\mathcal{Y}$ of the $(L, K)$-polarized generalized K3 surfaces contains a 19-dimensional subfamily of the K3 surfaces 
of the form $\{Y=(e^{\sqrt{-1}H}, \sigma^\vee)\}$, where $NS(M_{\sigma^\vee})=\Z H$ with $H^2=2n>0$. 
There is no K\"ahler deformation as explained in Section \ref{Kahler rigidity computation}.  
Indeed, for generic $X^\vee$, we have 
$$
L \cong \widetilde{NS}(Y) \subsetneq NS'(M_{\sigma^\vee}), \ \ \ K \cong \widetilde{T}(Y) \supsetneq T(M_{\sigma^\vee}). 
$$
The point is that $\sigma^\vee$ admits additional $1$-dimensional $B$-deformation as a $(L, K)$-polarized generalized K3 surface. 
\end{enumerate}

We therefore obtain two $20$-dimensional families of generalized K3 surfaces 
whose $A$-moduli space and $B$-moduli space are exchanged as expected (cf. Theorem \ref{moduli and dim theorem}).   
\begin{table}[htb]
\begin{center}
  \begin{tabular}{c|c|c}
    & \ $(K, L)$-polarization \ & \  $(L, K)$-polarization \   \\ \hline 
 $\mathfrak{M}_A$   & 20-dim   &   0-dim  \\ \hline
 $\mathfrak{M}_B$  & 0-dim   &   20-dim
  \end{tabular}
  \end{center}
\end{table}

The above discussion is readily generalized to any singular K3 surfaces. 
Let $S$ be a singular K3 surface such that $NS(S)=K'$ and $T(S)=L$, and set $K=K'\oplus U$. 
Then the mirror is given by a family $\mathcal{Y}$ of $(L,K)$-polarized generalized K3 surfaces.  
This family may include a K3 surface with $L_{e^{B+\sqrt{-1}\omega}} \cong L$, 
although it is normally very hard to construct such a K3 surface explicitly.  

In summary, the mirror partner of a singular K3 surfaces is a K\"ahler rigid generalized K3 surface, which is generally not a classical K3 surface. 
This solution addresses the problem of mirror symmetry for singular K3 surfaces.

\begin{Rem}
The initial motivation for the present article stems from the attractor mechanisms of the complex and K\"ahler moduli spaces of Calabi--Yau 3-folds \cite{FK}. 
The connection arises from a Calabi--Yau 3-fold formed by the product of an elliptic curve and a K3 surface. 
Then the complex (resp. K\"ahler) attractor mechanism specifies a set of isolated points in the moduli space, called the constellation, 
and they correspond to the complex (resp. symplectic) rigid K3 surfaces. 
For further elaboration, refer to \cite{FK}. 
\end{Rem}


\subsection{Comments on relevant works}  \label{comments}

A generalized Calabi--Yau variety is a Fano variety $X$ whose derived category $\mathrm{D}^b(X)$ admits a semi-orthogonal decomposition, 
with one component being the Calabi--Yau category $\mathcal{A}_X$. 
This category provides a deformation of the derived category of a Calabi--Yau manifold of lower dimension. 
In this sense, $\mathcal{A}_X$ is thought to be a non-commutative Calabi--Yau manifold. 

Another notion of a non-commutative Calabi--Yau manifold arises from a certain algebra introduced in \cite{Kan}; 
Here, the quotient category $\mathrm{qgr}(A/(f)) := \mathrm{gr}(A/(f))/\mathrm{tor}(A/(f))$ is considered,   
where $A$ is a specific choice of skew polynomial algebra and $f$ is a Fermat type element in the center of $A$. 
This construction yields a non-commutative projective scheme in the sense of Artin and Zhang \cite{AZ}. 

In dimension 2, there are thus at least three notions of non-commutative K3 surfaces: 
$\mathcal{A}_X$ for a Fano 4-fold $X$, the quotient category $\mathrm{qgr}(A/(f))$ of modules of a certain algebra, 
and  the derived category $\mathrm{D}^b(S,\alpha)$ of twisted coherent sheaves on a K3 surface $S$ (Remark \ref{twisted sheaves}). 
The relation between $\mathcal{A}_X$ and $\mathrm{D}^b(S,\alpha)$ has recently been elucidated, with a certain correspondence provided in \cite{AT, Huy2}. 
However, there exist instances of $\mathcal{A}_X$ that do not correspond to $\mathrm{D}^b(S,\alpha)$, and similarly, generalized Calabi--Yau structures that do not correspond to $\mathcal{A}_X$.  
The connection between the quotient category $\mathrm{qgr}(A/(f))$ and the other two cases remains unclear at this point. 

\begin{Rem} 
The author has been informed by Belmans that the methods in \cite{BBG} should generalize to higher dimensions, 
and show that the derived category $\mathrm{D}^b(\mathrm{qgr}(A/(f))$ of modules of the non-commutative K3 surfaces arising in \cite{Kan} is of the form $\mathrm{D}^b(S,\alpha)$. 
This work is  currently in progress by the authors of \cite{BBG}.
\end{Rem}

In \cite{SS} Sheridan and Smith proved, among other results, homological mirror symmetry for certain pairs of a rigid Calabi--Yau manifold and a generalized Calabi--Yau variety. 
They in particular discussed homological mirror symmetry for a rigid K3 surface and a cubic fourfold $X$.  
The corresponding $\mathcal{A}_X$ is considered ``non-geometric" and does not fall within the aforementioned correspondence. 
Nevertheless it should correspond to a generalized Calabi--Yau structure and align with our result. 
It is worth noting that our formulation is intrinsic and applicable to all singular K3 surfaces, 
whereas \cite{SS} focuses on (generalized) Calabi--Yau varieties obtained  through the Batyrev-Borisov construction (in all dimensions).


\par\noindent{\scshape \small
Department of Mathematics \\
Faculty of Science and Engineering\\
Waseda University \\
3-4-1 Okubo, Shinjuku, Tokyo, 169-8555, Japan}
\par\noindent{\ttfamily  a\_kanazawa@waseda.jp}


\begin{thebibliography}{99}
\bibitem{AT}N. Addington and R. Thomas, Hodge theory and derived categories of cubic fourfolds, Duke Math. J. 163, no. 10 (2014), 1885-1927. 
\bibitem{AZ}M. Artin and J. J. Zhang, Noncommutative Projective Schemes, Adv. Math. 109 (1994), no. 2, 228-287.
\bibitem{AM}P. Aspinwall and D. Morrison, String theory on K3 surfaces, Mirror symmetry, II, AMS/IP Stud. Adv. Math. 1 (1997), 703-716.
\bibitem{BB}V. Batyrev and L.Borisov, Dual cones and mirror symmetry for generalized Calabi--Yau manifolds, Mirror symmetry, II, AMS/IP Stud. Adv. Math., vol. 1, Amer. Math. Soc., Providence, RI, 1997, pp. 71-86. 
\bibitem{BBG}T. Baumann, P. Belmans and O. van Garderen, Central curves on noncommutative surfaces, arXiv:2410.07620
\bibitem{CDP} P. Candelas, E. Derrick, and L. Parkes, Generalized Calabi--Yau manifolds and the mirror of a rigid manifold, Nuclear Phys. B 407 (1993), no. 1, 115-154. 
\bibitem{Dol}I. Dolgachev, Mirror symmetry for lattice polarized K3 surfaces, J. Math. Sci. 81 (1996), no. 3, 2599-2630, Algebraic Geometry, 4.
\bibitem{FKY} Y.-W. Fan, A. Kanazawa and S.-T. Yau, Weil--Petersson geometry on the space of Bridgeland stability conditions, Comm. Anal. Geom. Vol. 29, No. 3, (2021) 681-706.
\bibitem{FK}Y.-W. Fan and A. Kanazawa, Attractor mechanisms of moduli spaces of Calabi--Yau 3-folds, J. Geom. Phys. Vol. 185 (2023) 104724. 
\bibitem{HK}K. Hashimoto and A. Kanazawa, Calabi--Yau threefolds of type K (II): Mirror Symmetry, Commun. Number Theory Phys. 10 (2) (2016) 157-192. 
\bibitem{HK2}S. Hosono and A. Kanazawa, BCOV cusp forms of lattice polarized K3 surfaces, Advances in Mathematics, Vol. 434 (2023) 109328. 
\bibitem{Hit}N. Hitchin, Generalized Calabi--Yau manifolds, Quart. J. Math. Oxford Ser. 54: 281-308, 2003. 
\bibitem{Huy}D. Huybrechts, Generalized Calabi--Yau structures, K3 surfaces, and $B$-fields, Int. J. Math. 16 (2005), 13-36. 
\bibitem{Huy2}D. Huybrechts, The K3 category of a cubic fourfold, Compos. Math. Vol. 153(3), 2017, 586-620. 
\bibitem{Kan}A. Kanazawa, Non-commutative projective Calabi--Yau schemes, JPAA, 2015; 219 (7) 2771-2780. 
\bibitem{Leu}C. Leung, Geometric Aspects of Mirror Symmetry (with SYZ for Rigid CY manifolds), Proceedings of ICCM, 2001. 
\bibitem{Sca}F. Scattone, On the compactification of moduli spaces for algebraic K3 surfaces, Mem. AMS 70 (1987) No.374.
\bibitem{SS}N. Sheridan and I. Smith, Homological mirror symmetry for generalized Greene--Plesser mirrors, Invent. Math. 224 (2021), no. 2, 627-682. 
\end{thebibliography}
\end{document}